\documentclass[10, amsmath]{amsart}
\usepackage{amsmath, amsthm, amscd, amsfonts, amssymb, graphicx, color}
\usepackage[bookmarksnumbered, colorlinks, plainpages]{hyperref}

\theoremstyle{plain}
\newtheorem{theorem}{Theorem}[section]
\newtheorem{corollary}[theorem]{Corollary}
\newtheorem{lemma}[theorem]{Lemma}

\theoremstyle{definition}
\newtheorem{definition}{Definition}[section]
\theoremstyle{notation}

\theoremstyle{remark}

\theoremstyle{note}

\numberwithin{equation}{section}

\begin{document}


\title[Existence of fixed points ...]{Existence of fixed points for pairs of mappings and application to Urysohn integral equations}

\author[D. Kumar]{Deepak Kumar$^1$}

\address{$^1$Department of Mathematics,
Lovely Professional University, Phagwara,
Punjab-144411, India.\\\newline
Research Scholar, IKG Punjab Technical University, Jalandhar-Kapurthala Highway, Kapurthala-144603, Punjab, India}
\email{\textcolor[rgb]{0.00,0.00,0.84}{deepakanand@live.in}}

\author[S. Chandok]{Sumit Chandok$^2$}


\address{$^2$School of Mathematics, Thapar University, Patiala-147004, Punjab, India.}
\email{\textcolor[rgb]{0.00,0.00,0.84}{sumit.chandok@thapar.edu, chandhok.sumit@gmail.com}}




\footnotetext{Preprint submitted to Journal ... \hfill \today}

\begin{abstract}
In this paper, we establish some common fixed point results for two pairs of weakly compatible mappings in the setting of $C$-complex valued metric space. Also, as application of the proved result, we obtain the existence and uniqueness of a common solution of the system of the Urysohn integral equations:
\begin{eqnarray*}
x(t)=\psi_i(t)+\int_{a}^{b}K_i(t,s,x(s))ds
\end{eqnarray*}
where $i=1, 2, 3, 4, a,b\in \mathbb{R}$ with $a\leq b, t\in [a,b], x, \psi_i\in C([a,b],\mathbb{R}^n)$ and $K_i:[a,b]\times [a,b]\times \mathbb{R}^n\rightarrow \mathbb{R}^n$ is a mapping for each $i=1, 2, 3, 4$.
\end{abstract}

\maketitle

\section{Introduction}
The notion of complex valued metric space which is more general than the metric space was introduced by Azam et at. (\cite{r1}). Afterwards many authors studied the existence and uniqueness of common fixed point and fixed point of self mappings with different contractive conditions in complex valued metric space (\cite{r1}-\cite{r4}, \cite{r5}-\cite{r7}, \cite{r10}, \cite{r11}). Recently, Sintunavarat et al. (\cite{r2}-\cite{r11}) introduced the concepts of $C$-Cauchy sequence and C-complete in complex-valued metric spaces and established the existence of common fixed point theorems in $C$-complete complex-valued metric spaces. They have applied the obtained results to check the existence and uniqueness of common solution of the integral equations.

The aim of this manuscript is to establish the common fixed point theorems for two pairs of weakly compatible mappings satisfying some contractive conditions in the framework of $C$-complex valued metric spaces. Also, as application of the proved results, we obtain the existence and uniqueness of a common solution of the system of the Urysohn integral equations:
\begin{eqnarray*}
x(t)=\psi_i(t)+\int_{a}^{b}K_i(t,s,x(s))ds
\end{eqnarray*}
where $i=1, 2, 3, 4, a,b\in \mathbb{R}$ with $a\leq b, t\in [a,b], x, \psi_i\in C([a,b],\mathbb{R}^n)$ and $K_i:[a,b]\times [a,b]\times \mathbb{R}^n\rightarrow \mathbb{R}^n$ is a mapping for each $i=1, 2, 3, 4$.\\

\section{Preliminaries}

Let $\mathbb{C}$ be the set of complex numbers. For $z_1, z_2 \in \mathbb{C},$ we will define a partial order $\preceq $ on $\mathbb{C}$ as follows:
\begin{eqnarray*}
z_1\preceq z_2 \Longleftrightarrow Re(z_1)\leq Re(z_2) \textit{ and } Im(z_1)\leq Im(z_2)
\end{eqnarray*}
We note that $z_1\preceq z_2$ if one of the following holds:\\
$(C_1)$ $Re(z_1)=Re(z_2)$ and $Im(z_1)=Im(z_2)$;\\
$(C_2)$ $Re(z_1)<Re(z_2)$ and $Im(z_1)=Im(z_2)$;\\
$(C_3)$ $Re(z_1)=Re(z_2)$ and $Im(z_1)<Im(z_2)$;\\
$(C_4)$ $Re(z_1)<Re(z_2)$ and $Im(z_1)<Im(z_2)$;\\

It is obvious that if $a,b\in \mathbb{R}$ such that $a\preceq b$, then $az\preceq bz$ for all $z\in \mathbb{C}$.\\

In particular, we write $z_1\precnsim z_2$ if $z_1\neq z_2$ and one of $(C_2), (C_3), (C_4)$ is satisfied, and we write $z_1\prec z_2$ if and only if $(C_4)$ is satisfied. The following are well known:\\
\begin{eqnarray*}
0\preceq z_1 \precnsim z_2 \Rightarrow |z_1|<|z_2|,
\end{eqnarray*}
\begin{eqnarray*}
z_1\preceq z_2, z_2\prec z_3 \Rightarrow z_1\prec z_3
\end{eqnarray*}

The following definitions and results will be needed in the sequel.

\begin{definition}\cite{r1}
Let $X$ be a nonempty set. Suppose that the mapping $d:X\times X\rightarrow \mathbb{C}$
satisfies the following conditions:\\
$(i)$ $0\preceq d(x,y)$ for all $x,y\in X$ and $d(x,y)=0$ if and only if $x=y$;\\
$(ii)$ $d(x,y)=d(y,x)$ for all $x,y\in X$;\\
$(iii)$ $d(x,y)\preceq d(x,z)+d(z,y)$ for all $x,y,z \in X$.\\
Then $d$ is called a complex-valued metric on $X$ and $(X,d)$ is called a complex valued metric space.
\end{definition}

\begin{definition}\cite{r1}
Let $(X,d)$ be a complex-valued metric space.\\
$(1):$ A point $x\in X$ is called an interior point of a set $A\subseteq X$ wherever there exits $0\prec r\in \mathbb{C}$ such that
\begin{eqnarray*}
B(x,r)=\{y\in X: d(x,y)\prec r\}\subseteq A.
\end{eqnarray*}
$(2):$ A point $x\in X$ is called a limit point of $A$ whenever, for all $0\prec r\in \mathbb{C}$,
\begin{eqnarray*}
B(x,r)\cap (A-X)\neq \varnothing
\end{eqnarray*}
$(3):$ A set $A\subseteq X$ is called open whenever each element of $A$ is an interior point of $A.$\\
$(4):$ A set $A\subset X$ is called closed whenever each limit point of $A$ belongs to $A.$\\
$(5):$ A sub-basis for a Hausdorff topology $\tau$ on $X$ is the family
\begin{eqnarray*}
F=\{B(x,r): x\in X \textit{ and } 0\prec r\}.
\end{eqnarray*}
\end{definition}

For more concepts such as limit point, convergence and Cauchy sequence in complex valued metric spaces (see \cite{r1,r4,r2} and references cited therein).


\begin{definition}\cite{r2}
	Let $(X,d)$ be a complex-valued metric space and $\{x_n\}$ be a sequence in $X$.\\
	$(i)$ If, for any $c\in \mathbb{C}$ with $0\prec c$, there exists $N\in \mathbb{N}$ such that, for all $m,n>N$, $d(x_n,x_m)\prec c$, then $\{x_n\}$ is called a C-Cauchy sequence in $X$.\\
	$(ii)$ If every C-Cauchy sequence in $X$ is convergent, then $(X,d)$ is said to be a $C$-complete complex valued metric space.
\end{definition}
For more concepts in $C$-complete complex valued metric space see \cite{r2}.

\begin{definition}
	Let $S$ and $T$ be self mappings of a nonempty set $X.$\\
	$(1):$ A point $x\in X$ is called a fixed point of $T$ if $Tx=x.$\\
	$(2):$ A point $x\in X$ is called a coincidence point of $S$ and $T$ if $Sx=Tx$ and the point $w\in X$ such that $w=Sx=Tx$ is called a point of coincidence of $S$ and $T.$\\
	$(3):$ A point $x\in X$ is called a common fixed point of $S$ and $T$ if $x=Sx=Tx.$
\end{definition}

\begin{definition}\cite{r12}
Let $X$ be a complex valued metric space. Then a pair of self mappings $S,T:X\rightarrow X$ is said to be weakly compatible if they commute at their coincidence points.
\end{definition}

The following results of \cite{r1} will be used in the sequel.

\begin{lemma}\cite{r1}
Let $(X,d)$ be a complex valued metric space, $\{x_n\}$ be a sequence in $X.$ Then $\{x_n\}$ converges to a point $x\in X$ if and only if $|d(x_n,x)|\rightarrow 0$ as $n\rightarrow \infty.$
\end{lemma}

\begin{lemma}\cite{r1}
Let $(X,d)$ be a complex valued metric space, $\{x_n\}$ be a sequence in $X.$ Then $\{x_n\}$ is a Cauchy sequence if and only if $|d(x_n,x_{n+m})|\rightarrow 0 $ as $n\rightarrow \infty $, where $m\in \mathbb{N}$.
\end{lemma}

\section{Main Results}
Throughout this paper, $\mathbb{R}$ denotes a set of real numbers, $\mathbb{C}_+$ denotes the set $\{c\in \mathbb{C}:0\preceq c\}$ and
$\Gamma$ denotes the class of all functions $\gamma:\mathbb{C}_+\times \mathbb{C}_+\rightarrow [0,1)$ which satisfies the following condition:
\begin{center}
	for any $(x_n,y_n)$ in $\mathbb{C}_+\times\mathbb{C}_+$, $\gamma(x_n,y_n)\rightarrow 1 \Rightarrow (x_n,y_n)\rightarrow 0$.
\end{center}

Now, we prove the main results.

\begin{theorem}\label{thm31}
Let $(X,d)$ be a C-complete complex valued metric space and $f,g,S,T:X\rightarrow X$ be four self mappings. If there exists three mappings $\lambda_1$, $\lambda_2$, $\lambda_3 : \mathbb{C}_+\times\mathbb{C}_+  \rightarrow [0,1)$, such that for all $x,y\in X$ the following axioms holds:\\
\begin{enumerate}
  \item  $\lambda_1(x,y)+2\lambda_2(x,y)+2\lambda_3(x,y)<1$;\\
  \item the mapping $\gamma : \mathbb{C}_+\times\mathbb{C}_+\rightarrow [0,1)$ defined by $\gamma(x,y)=\frac{\lambda_1(x,y)+\lambda_2(x,y)+\lambda_3(x,y)}{1-\lambda_2(x,y)-\lambda_3(x,y)}$ belongs to $\Gamma$;\\
  \item $d(Sx,Ty)\preceq \lambda_1(fx,gy)d(fx,gy)+\lambda_2(fx,gy)(d(fx,Sx)+d(gy,Ty))+\lambda_3(fx,gy)(d(fx,Ty)+d(gy,Sx))$;\\
  \item $T(X)\subseteq f(X)$ and $S(X)\subseteq g(X)$ and the pairs $(f,S)$ and $(g,T)$ are weakly compatible\\
\end{enumerate}

Then $f$, $g$, $S$ and $T$ have a unique common fixed point in $X$.

\end{theorem}
\begin{proof}
Let $x_0$ be an arbitrary point in $X$. Since $T(X)\subseteq f(X)$ and $S(X)\subseteq g(X)$, we construct the two sequences $\{x_n\}$ and $\{y_n\}$  in $X$ such that
\begin{eqnarray}\label{eq0}
Sx_{2n-2}=gx_{2n-1}=y_{2n-1} \textit{ and } Tx_{2n-1}=fx_{2n}=y_{2n} \quad \text{for all}\quad n\geq 0.
\end{eqnarray}
For $n\geq 0$, we get
\begin{eqnarray*}
d(y_{2n+1},y_{2n+2})&=&d(Sx_{2n},Tx_{2n+1})\\
&\preceq&\lambda_1(fx_{2n},gx_{2n+1})d(fx_{2n},gx_{2n+1})+\lambda_2(fx_{2n},gx_{2n+1})(d(fx_{2n},Sx_{2n})+d(gx_{2n+1},Tx_{2n+1}))\\
&+&\lambda_3(fx_{2n},gx_{2n+1})(d(fx_{2n},Tx_{2n+1})+d(gx_{2n+1},Sx_{2n}))\\
&=&
\lambda_1(y_{2n},y_{2n+1})d(y_{2n},y_{2n+1})+\lambda_2(y_{2n},y_{2n+1})(d(y_{2n},y_{2n+1})+d(y_{2n+1},y_{2n+2}))\\
&+&\lambda_3(y_{2n},y_{2n+1})(d(y_{2n},y_{2n+2})+d(y_{2n+1},y_{2n+1}))\\
&=& (\lambda_1(y_{2n},y_{2n+1})+\lambda_2(y_{2n},y_{2n+1})+\lambda_3(y_{2n},y_{2n+1}))d(y_{2n},y_{2n+1})\\
&+&(\lambda_2(y_{2n},y_{2n+1})+\lambda_3(y_{2n},y_{2n+1}))d(y_{2n+1},y_{2n+2}).
\end{eqnarray*}
This implies that
\begin{equation*}
(1-(\lambda_2(y_{2n},y_{2n+1})+\lambda_3(y_{2n},y_{2n+1})))d(y_{2n+1},y_{2n+2})\preceq (\lambda_1(y_{2n},y_{2n+1})+\lambda_2(y_{2n},y_{2n+1})+\lambda_3(y_{2n},y_{2n+1}))d(y_{2n},y_{2n+1})
\end{equation*}
Further, we have
\begin{eqnarray*}
|d(y_{2n+1},y_{2n+2})|&\preceq& \frac{\lambda_1(y_{2n},y_{2n+1})+\lambda_2(y_{2n},y_{2n+1})+\lambda_3(y_{2n},y_{2n+1})}{1-(\lambda_2(y_{2n},y_{2n+1})+\lambda_3(y_{2n},y_{2n+1}))}|d(y_{2n},y_{2n+1})|\\
&=& \gamma(y_{2n},y_{2n+1})|d(y_{2n},y_{2n+1})|\quad  \text{for all}\quad  n\in \mathbb{N}.
\end{eqnarray*}
On the similar lines, we obtain that
\begin{eqnarray*}
|d(y_{2n},y_{2n+1})|\leq \gamma(y_{2n-1},y_{2n})|d(y_{2n-1},y_{2n})| \quad \text{for all}\quad  n\in \mathbb{N}
\end{eqnarray*}
Consequently, we get
\begin{eqnarray}\label{eq1}
|d(y_{n},y_{n+1})|\leq \gamma(y_{n-1},y_{n})|d(y_{n-1},y_{n})|\quad  \text{for all}\quad  n\in \mathbb{N}
\end{eqnarray}
Thus the sequence $\{|d(y_{n},y_{n+1})|\}_{n\in \mathbb{N}}$ is monotone non increasing and bounded below. Therefore,
$|d(y_{n},y_{n+1})|\rightarrow l$ for some $l\geq0$. Now, we claim that $l=0$. On contrary, assume that $l>0$. Then by taking
limit as $n\rightarrow \infty$ in $(\ref{eq1})$, we get
\begin{eqnarray*}
|l|\leq \lim_{n\rightarrow \infty}\gamma(y_{n-1},y_{n})|l|\leq|l|
\end{eqnarray*}
\begin{eqnarray*}
\Rightarrow \lim_{n\rightarrow \infty}\gamma(y_{n-1},y_{n})\rightarrow 1
\end{eqnarray*}
But $\gamma \in \Gamma$, which implies that $(y_{n-1},y_{n})\rightarrow 0$, that is $d(y_{n-1},y_{n})\rightarrow 0$, which is contradiction to the fact that $l>0$. Thus $l=0$ and hence,
\begin{eqnarray}\label{eq2}
\lim_{n\rightarrow\infty}|d(y_{n-1},y_{n})|=0
\end{eqnarray}

Next, we shall prove that $\{y_n\}$ is a $C$-Cauchy sequence, to prove this it is enough to show that the subsequence $\{y_{2n}\}$ is a $C$-Cauchy sequence.
On contrary, suppose that $\{y_{2n}\}$ is not a $C$-Cauchy sequence. Then there exist $c\in C$ with $0\prec c$ for which, for all $k\in\mathbb{N}$ there exists $2m_k>2n_k\geq k$ such that
\begin{eqnarray}\label{eq3}
d(y_{2n_k},y_{2m_k})\succeq c
\end{eqnarray}
Now, corresponding to $n_k$, we can choose $m_k$ in such a way that it is the smallest integer with
$2m_k>2n_k\geq k$ satisfying (\ref{eq3}). Then
\begin{eqnarray}\label{eq4}
d(y_{2m_k},y_{2m_k-2})\prec c
\end{eqnarray}
From equation (\ref{eq3}), (\ref{eq4}) and triangular inequality, we have
\begin{eqnarray*}
c \preceq d(y_{2n_k},y_{2m_k}) &\preceq& d(y_{2n_k},y_{2m_k-2})+d(y_{2m_k-2},y_{2m_k-1})+d(y_{2m_k-1},y_{2m_k}) \\
  &\prec & c+d(y_{2m_k-2},y_{2m_k-1})+d(y_{2m_k-1},y_{2m_k})
\end{eqnarray*}
which implies that
\begin{eqnarray*}
|c|\leq |d(y_{2n_k},y_{2m_k})|\leq|c|+|d(y_{2m_k-2},y_{2m_k-1})|+|d(y_{2m_k-1},y_{2m_k})|
\end{eqnarray*}
Taking limit as $k\rightarrow \infty$ and using (\ref{eq2}), we have
\begin{eqnarray*}
  |c|\leq \lim_{k\rightarrow \infty}|d(y_{2n_k},y_{2m_k})|\leq |c|
\end{eqnarray*}
which implies that
\begin{eqnarray}\label{eq5}
\lim_{k\rightarrow \infty}|d(y_{2n_k},y_{2m_k})|=|c|
\end{eqnarray}
Now, using triangular inequality, we have
\begin{eqnarray*}
|d(y_{2n_k},y_{2m_k})|&\leq & |d(y_{2n_k},y_{2m_k+1})|+|d(y_{2m_k+1},y_{2m_k})|\\
&\leq &|d(y_{2n_k},y_{2m_k})|+|d(y_{2m_k},y_{2m_k+1})|+|d(y_{2m_k+1},y_{2m_k})|
\end{eqnarray*}
Taking limit as $k\rightarrow \infty$ and using (\ref{eq2}) and (\ref{eq5}), we get
\begin{eqnarray*}
|c|\leq \lim_{k\rightarrow\infty}|d(y_{2n_k},y_{2m_k+1})|\leq |c|
\end{eqnarray*}
\begin{eqnarray}\label{eq6}
\lim_{k\rightarrow\infty}|d(y_{2n_k},y_{2m_k+1})|=|c|
\end{eqnarray}
Next, we have
\begin{eqnarray*}
 d(y_{2n_k},y_{2m_k}) &\preceq & d(y_{2n_k},y_{2n_k+1})+d(y_{2n_k+1},y_{2m_k+2})+d(y_{2m_k+2},y_{2m_k+1})\\
  &=& d(y_{2n_k},y_{2n_k+1})+d(Sx_{2n_k},Tx_{2m_k+1})+d(y_{2m_k+2},y_{2m_k+1})\\
  &\preceq & d(y_{2n_k},y_{2n_k+1})+\lambda_1(fx_{2n_k},gx_{2m_k+1})d(fx_{2n_k},gx_{2m_k+1})+\lambda_2(fx_{2n_k},gx_{2m_k+1})(d(fx_{2n_k},Sx_{2n_k})\\
  &+& d(gx_{2m_k+1},Tx_{2m_k+1}))+\lambda_3(fx_{2n_k},gx_{2m_k+1})(d(fx_{2n_k},Tx_{2m_k+1})+d(gx_{2m_k+1},Sx_{2n_k}))\\
  &+& d(y_{2m_k+2},y_{2m_k+1})\\
  &= & d(y_{2n_k},y_{2n_k+1})+\lambda_1(y_{2n_k},y_{2m_k+1})d(y_{2n_k},y_{2m_k+1})+\lambda_2(y_{2n_k},y_{2m_k+1})(d(y_{2n_k},y_{2n_k+1})\\
  &+& d(y_{2m_k+1},y_{2m_k+2}))+\lambda_3(y_{2n_k},y_{2m_k+1})(d(y_{2n_k},y_{2m_k+2})+d(y_{2m_k+1},y_{2n_k+1}))\\
  &+& d(y_{2m_k+2},y_{2m_k+1})
\end{eqnarray*}
which implies,
\begin{eqnarray*}
 |d(y_{2n_k},y_{2m_k})|&\leq&|d(y_{2n_k},y_{2n_k+1})|+\lambda_1(y_{2n_k},y_{2m_k+1})|d(y_{2n_k},y_{2m_k+1})|+\lambda_2(y_{2n_k},y_{2m_k+1})|(d(y_{2n_k},y_{2n_k+1})\\
  &+& d(y_{2m_k+1},y_{2m_k+2}))|+\lambda_3(y_{2n_k},y_{2m_k+1})|(d(y_{2n_k},y_{2m_k+2})+d(y_{2m_k+1},y_{2n_k+1}))|+|d(y_{2m_k+2},y_{2m_k+1})|\\
  &\leq& |d(y_{2n_k},y_{2n_k+1})|+(\lambda_1(y_{2n_k},y_{2m_k+1})+\lambda_2(y_{2n_k},y_{2m_k+1})+\lambda_3(y_{2n_k},y_{2m_k+1}))|d(y_{2n_k},y_{2m_k+1})|\\
  &+& |d(y_{2n_k},y_{2n_k+1})+d(y_{2m_k+1},y_{2m_k+2})|+|d(y_{2n_k},y_{2m_k+2})+d(y_{2m_k+1},y_{2n_k+1})+d(y_{2m_k+2},y_{2m_k+1})|\\
&\leq&|d(y_{2n_k},y_{2n_k+1})|+\frac{(\lambda_1(y_{2n_k},y_{2m_k+1})+\lambda_2(y_{2n_k},y_{2m_k+1})+\lambda_3(y_{2n_k},y_{2m_k+1}))}{1-(\lambda_2(y_{2n_k},y_{2m_k+1})+\lambda_3(y_{2n_k},y_{2m_k+1}))}
  |d(y_{2n_k},y_{2m_k+1})|\\
  &+& |d(y_{2n_k},y_{2n_k+1})+d(y_{2m_k+1},y_{2m_k+2})|+|d(y_{2n_k},y_{2n_k+1})|\\
 &=&|d(y_{2n_k},y_{2n_k+1})|+\gamma(y_{2n_k},y_{2m_k+1})|d(y_{2n_k},y_{2m_k+1})|+|d(y_{2n_k},y_{2n_k+1})+d(y_{2m_k+1},y_{2m_k+2})|\\
 &+&|d(y_{2n_k},y_{2n_k+1})|\\
 &\leq&|d(y_{2n_k},y_{2n_k+1})|+|d(y_{2n_k},y_{2m_k+1})|+|d(y_{2n_k},y_{2n_k+1})+d(y_{2m_k+1},y_{2m_k+2})|+|d(y_{2n_k},y_{2n_k+1})|\\
\end{eqnarray*}
Taking limit as $k\rightarrow\infty$ and using equation (\ref{eq2}) and (\ref{eq6}), we get
\begin{eqnarray*}
|c|\leq \lim_{n\rightarrow\infty}\gamma(y_{2n_k},y_{2m_k+1})|c|\leq|c|
\end{eqnarray*}
therefore,
\begin{eqnarray*}
\lim_{n\rightarrow\infty}\gamma(y_{2n_k},y_{2m_k+1})=1
\end{eqnarray*}
Since $\gamma\in \Gamma$,  we have $d(y_{2n_k},y_{2m_k+1})\rightarrow 0$ as $k\rightarrow\infty$, which is contradiction. Thus $\{y_{2n}\}$ is a $C$-Cauchy sequence and hence $\{y_n\}$ is a $C$-Cauchy sequence. As $X$ is C-complete, therefore there exists $t\in X$ such that $y_n\rightarrow t$ as $n\rightarrow\infty$, Therefore from equation (\ref{eq0}), we get
\begin{eqnarray}\label{eq7}
\lim_{n\rightarrow\infty}Sx_{2n}=\lim_{n\rightarrow\infty}Tx_{2n+1}=\lim_{n\rightarrow\infty}fx_{2n}=\lim_{n\rightarrow\infty}gx_{2n+1}=t
\end{eqnarray}
Next, since $S(X)\subseteq g(X)$, there exist $u\in X$ such that $g(u)=t$. Thus equation (\ref{eq7}) becomes
\begin{eqnarray}\label{eq8}
\lim_{n\rightarrow\infty}Sx_{2n}=\lim_{n\rightarrow\infty}Tx_{2n+1}=\lim_{n\rightarrow\infty}fx_{2n}=\lim_{n\rightarrow\infty}gx_{2n+1}=t=g(u)
\end{eqnarray}
We shall show that $Tu=gu$, for this consider
\begin{eqnarray*}
d(t,Tu)&\preceq& d(t,Sx_{2n})+d(Sx_{2n},Tu)\\
&\preceq& d(t,Sx_{2n})+ \lambda_1(fx_{2n},gu)d(fx_{2n},gu)+\lambda_2(fx_{2n},gu)(d(fx_{2n},Sx_{2n})+d(gu,Tu))+\lambda_3(fx_{2n},gu)(d(fx_{2n},Tu)\\
&+&d(gu,Sx_{2n}))
\end{eqnarray*}
Taking limit as $k\rightarrow\infty$ and using equation (\ref{eq8}), we get
\begin{eqnarray*}
d(t,Tu)&\preceq& d(t,t)+ \lambda_1(gu,gu)d(gu,gu)+\lambda_2(gu,gu)(d(gu,gu)+d(gu,Tu))+\lambda_3(gu,gu)(d(gu,Tu)+d(gu,gu))\\
&=&\lambda_2(gu,gu)d(gu,Tu)+\lambda_3(gu,gu)(d(gu,Tu)
\end{eqnarray*}
this implies,
\begin{eqnarray*}
\bigg(1-(\lambda_2(gu,gu)+\lambda_3(gu,gu))\bigg)d(t,Tu)\preceq 0
\end{eqnarray*}
therefore,
\begin{eqnarray*}
\bigg(1-(\lambda_2(gu,gu)+\lambda_3(gu,gu))\bigg)|d(t,Tu)|\leq 0,
\end{eqnarray*}
which implies
\begin{eqnarray}\label{eq9}
d(t,Tu)=0, \quad\text{therefore}\quad Tu=t \quad\text{and hence}\quad Tu=gu=t
\end{eqnarray}
Also it is given that $T(X)\subseteq f(X)$, therefore there exist $v\in X$ such that $fv=t$. Thus from equation (\ref{eq7}), we get
\begin{eqnarray}\label{eq10}
\lim_{n\rightarrow\infty}Sx_{2n}=\lim_{n\rightarrow\infty}Tx_{2n+1}=\lim_{n\rightarrow\infty}fx_{2n}=\lim_{n\rightarrow\infty}gx_{2n+1}=t=fv
\end{eqnarray}
Now, we shall show that $Sv=fv$. For this consider
\begin{eqnarray*}
d(Sv,t)&\preceq& d(Sv,Tx_{2n+1})+d(Tx_{2n+1},t)\\
&\preceq&\lambda_1(fv,gx_{2n+1})d(fv,gx_{2n+1})+\lambda_2(fv,gx_{2n+1})(d(fv,Sv)+d(gx_{2n+1},Tx_{2n+1}))+\lambda_3(fv,gx_{2n+1})(d(fv,Tx_{2n+1})\\
&+&d(gx_{2n+1},Sv))+d(Tx_{2n+1},t)
\end{eqnarray*}
Taking limit as $k\rightarrow\infty$ and using equation (\ref{eq10}), we get
\begin{eqnarray*}
d(Sv,t)&\preceq&\lambda_1(fv,fv)d(fv,t)+\lambda_2(fv,fv)(d(fv,Sv)+d(t,t)+\lambda_3(fv,fv)(d(fv,t)+d(t,Sv))+d(t,t)\\
&\preceq&\lambda_1(fv,fv)d(t,t)+\lambda_2(fv,fv)(d(t,Sv)+d(t,t)+\lambda_3(fv,fv)(d(t,t)+d(t,Sv))+d(t,t)\\
&=&(\lambda_2(fv,fv)+\lambda_3(fv,fv))d(t,Sv)
\end{eqnarray*}
this implies,
\begin{eqnarray*}
\bigg(1-(\lambda_2(fv,fv)+\lambda_3(fv,fv))\bigg)d(t,Sv)\preceq 0
\end{eqnarray*}
therefore,
\begin{eqnarray*}
\bigg(1-(\lambda_2(fv,fv)+\lambda_3(fv,fv))\bigg)|d(t,Sv)|\leq 0
\end{eqnarray*}
which implies $d(t,Sv)=0$, therefore $t=Sv$ and hence from equation (\ref{eq10}), we get
\begin{eqnarray}\label{eq11}
Sv=fv=t
\end{eqnarray}
Therefore from equations (\ref{eq9}) and (\ref{eq11}), we get
\begin{eqnarray}\label{eq12}
Tu=gu=Sv=fv=t
\end{eqnarray}
Since the pairs $(f,S)$ and $(g,T)$ are weakly compatible. Therefore from equation (\ref{eq12}), we have
\begin{eqnarray}\label{eq13}
fv=Sv \Rightarrow Sfv=fSv \Rightarrow St=ft
\end{eqnarray}
and
\begin{eqnarray}\label{eq14}
gu=Tu \Rightarrow Tgu=gTu \Rightarrow Tt=gt
\end{eqnarray}
which implies that $t$ is a coincident point of each pair $(f,S)$ and $(g,T)$ in $X$.\\
Next, we shall show that $t$ is common fixed of $S, T, f$ and $g$. For this assume that $St=t$, if not then using the condition (3) of Theorem 3.1
with $x=t$ and $y=u$, we have
\begin{eqnarray*}
d(St,Tu)\preceq \lambda_1(ft,gu)d(ft,gu)+\lambda_2(ft,gu)(d(ft,St)+d(gu,Tu))+\lambda_3(ft,gu)(d(ft,Tu)+d(gu,St))
\end{eqnarray*}
Using equations (\ref{eq12}) and (\ref{eq13}), we have
\begin{eqnarray*}
d(St,t)&\preceq& \lambda_1(St,t)d(St,t)+\lambda_2(St,t)(d(St,St)+d(t,t))+\lambda_3(St,t)(d(St,t)+d(t,St))\\
&\preceq&\lambda_1(St,t)d(St,t)+\lambda_3(St,t)2d(St,t)\\
&=&(\lambda_1(St,t)+2\lambda_3(St,t))d(St,t)
\end{eqnarray*}
which implies,
\begin{eqnarray*}
(1-(\lambda_1(St,t)+2\lambda_3(St,t)))d(St,t)\preceq 0
\end{eqnarray*}
therefore,
\begin{eqnarray*}
(1-(\lambda_1(St,t)+2\lambda_3(St,t)))|d(St,t)|\leq 0
\end{eqnarray*}
Hence $d(St,t)=0$ i.e $St=t$, therefore from equation (\ref{eq13}), we get
\begin{eqnarray}\label{eq15}
ft=St=t
\end{eqnarray}
Similarly assume that $Tt=t$, if not then using the condition (3) of Theorem 3.1
with $x=v$ and $y=t$, we have $Tt=t$, therefore from equation (\ref{eq14}), we get
\begin{eqnarray}\label{eq16}
gt=Tt=t
\end{eqnarray}
From equations (\ref{eq15}) and (\ref{eq16}), we have
\begin{eqnarray*}
ft=gt=St=Tt=t
\end{eqnarray*}
Thus $t$ is a common fixed point of $S,T,f$ and $g$.\\
To check uniqueness, assume that $t^*\neq t$ be another fixed point of $S,T,f$ and $g$. Let $x=t$ and $y=t^*$ in condition (3) of Theorem 3.1, we get
\begin{eqnarray*}
d(t,t^*)=d(St,Tt^*)&\preceq&\lambda_1(ft,gt^*)d(ft,gt^*)+\lambda_2(ft,gt^*)(d(ft,St)+d(gt^*,Tt^*))+\lambda_3(ft,gt^*)(d(ft,Tt^*)+d(gt^*,St))\\
&\preceq&\lambda_1(t,t^*)d(t,t^*)+\lambda_2(t,t^*)(d(t,t)+d(t^*,t^*))+\lambda_3(t,t^*)(d(t,t^*)+d(t^*,t))\\
&=&(\lambda_1(t,t^*)+2\lambda_3(t,t^*))d(t,t^*),
\end{eqnarray*}
This implies
\begin{eqnarray*}
\bigg(1-(\lambda_1(t,t^*)+2\lambda_3(t,t^*))\bigg)d(t,t^*)\preceq 0
\end{eqnarray*}
therefore,
\begin{eqnarray*}
\bigg(1-(\lambda_1(t,t^*)+2\lambda_3(t,t^*))\bigg)|d(t,t^*)|\leq 0
\end{eqnarray*}
which implies $d(t,t^*)=0$ as $\lambda_1(x,y)+2\lambda_2(x,y)+2\lambda_3(x,y)<1$, hence $t=t^*$,
thus $t$ is a unique common fixed point of $S,T,f$ and $g$.
\end{proof}

Using the same arguments as in Theorem \ref{thm31}, we have the following result.

\begin{theorem}
Let $(X,d)$ be a C-complete complex valued metric space and $f,g,S,T:X\rightarrow X$ be four self mappings. If there exists three mappings $\lambda_1$, $\lambda_2$, $\lambda_3 : \mathbb{C}_+\times\mathbb{C}_+  \rightarrow [0,1)$ such that for all $x,y\in X$ the following axioms holds:\\
\begin{enumerate}

  \item  $\lambda_1(x,y)+\lambda_2(x,y)+\lambda_3(x,y)<1$;\\
  \item the mapping $\gamma : \mathbb{C}_+\times\mathbb{C}_+\rightarrow [0,1)$ defined by $\gamma(x,y)=\frac{\lambda_1(x,y)}{1-\lambda_2(x,y)-\lambda_3(x,y)}$ belongs to $\Gamma$;\\
  \item $d(Sx,Ty)\preceq \lambda_1(fx,gy)d(fx,gy)+\lambda_2(fx,gy)\frac{d(Sx,fx)d(Ty,gy)}{1+d(fx,gy)}+\lambda_3(fx,gy)\frac{d(Sx,gy)d(Ty,gy)}{1+d(fx,gy)}$;\\
  \item $T(X)\subseteq f(X)$ and $S(X)\subseteq g(X)$ and the pairs $(f,S)$ and $(g,T)$ are weakly compatible\\
\end{enumerate}

Then $f$, $g$, $S$ and $T$ have a unique common fixed point in $X$.

\end{theorem}

\begin{corollary}
Let $(X,d)$ be a C-complete complex valued metric space and $f,g,S,T:X\rightarrow X$ be four self mappings satisfying the following conditions;\\
\begin{enumerate}
  \item $d(Sx,Ty)\preceq \lambda_1d(fx,gy)+\lambda_2(d(fx,Sx)+d(gy,Ty))+\lambda_3(d(fx,Ty)+d(gy,Sx))$ for all
   $x,y\in X$, where $\lambda_1, \lambda_2, \lambda_3\in R^+$ with $\lambda_1+2\lambda_2+2\lambda_3<1$; \\
  \item $T(X)\subseteq f(X)$ and $S(X)\subseteq g(X)$ and the pairs $(f,S)$ and $(g,T)$ are weakly compatible\\
\end{enumerate}
Then $f$, $g$, $S$ and $T$ have a unique common fixed point in $X$.
\end{corollary}
\begin{corollary}
Let $(X,d)$ be a C-complete complex valued metric space and $S,T:X\rightarrow X$ be two self mappings. If there exists three mappings $\lambda_1$, $\lambda_2$, $\lambda_3 : \mathbb{C}_+\times\mathbb{C}_+  \rightarrow [0,1)$ for all $x,y\in X$ satisfying the following conditions;\\
\begin{enumerate}
  \item  $\lambda_1(x,y)+2\lambda_2(x,y)+2\lambda_3(x,y)<1$ for all $(x,y)\in \mathbb{C}_+\times\mathbb{C}_+$;\\
  \item the mapping $\gamma : \mathbb{C}_+\times\mathbb{C}_+\rightarrow [0,1)$ defined by $\gamma(x,y)=\frac{\lambda_1(x,y)+\lambda_2(x,y)+\lambda_3(x,y)}{1-\lambda_2(x,y)-\lambda_3(x,y)}$ belongs to $\Gamma$;\\
  \item $d(Sx,Ty)\preceq \lambda_1(x,y)d(x,y)+\lambda_2(x,y)(d(x,Sx)+d(y,Ty))+\lambda_3(x,y)(d(x,Ty)+d(y,Sx))$\\
\end{enumerate}
Then $S$ and $T$ have a unique common fixed point in $X$.
\end{corollary}
\begin{corollary}
Let $(X,d)$ be a C-complete complex valued metric space and $T:X\rightarrow X$ be a self mappings. If there exists three mappings $\lambda_1$, $\lambda_2$, $\lambda_3 : \mathbb{C}_+\times\mathbb{C}_+  \rightarrow [0,1)$ for all $x,y\in X$ satisfying the following conditions;\\
\begin{enumerate}
  \item  $\lambda_1(x,y)+2\lambda_2(x,y)+2\lambda_3(x,y)<1$ for all $(x,y)\in \mathbb{C}_+\times\mathbb{C}_+$;\\
  \item the mapping $\gamma : \mathbb{C}_+\times\mathbb{C}_+\rightarrow [0,1)$ defined by $\gamma(x,y)=\frac{\lambda_1(x,y)+\lambda_2(x,y)+\lambda_3(x,y)}{1-\lambda_2(x,y)-\lambda_3(x,y)}$ belongs to $\Gamma$;\\
  \item $d(Tx,Ty)\preceq \lambda_1(x,y)d(x,y)+\lambda_2(x,y)(d(x,Tx)+d(y,Ty))+\lambda_3(x,y)(d(x,Ty)+d(y,Tx))$\\
\end{enumerate}
Then $T$ have a unique fixed point in $X$.
\end{corollary}
\begin{corollary}
Let $(X,d)$ be a C-complete complex valued metric space and $f,g,S,T:X\rightarrow X$ be four self mappings satisfying the following conditions;
\begin{enumerate}
  \item $d(Sx,Ty)\preceq \lambda_1d(fx,gy)+\lambda_2\frac{d(Sx,fx)d(Ty,gy)}{1+d(fx,gy)}+\lambda_3\frac{d(Sx,gy)d(Ty,gy)}{1+d(fx,gy)}$;
    for all $x,y\in X$, where\\
   $\lambda_1, \lambda_2, \lambda_3\in R^+$ with $\lambda_1+\lambda_2+\lambda_3<1$;\\
  \item $T(X)\subseteq f(X)$ and $S(X)\subseteq g(X)$ and the pairs $(f,S)$ and $(g,T)$ are weakly compatible
\end{enumerate}
Then $f$, $g$, $S$ and $T$ have a unique common fixed point in $X$.
\begin{corollary}
Let $(X,d)$ be a C-complete complex valued metric space and $S,T:X\rightarrow X$ be two self mappings. If there exists three mappings $\lambda_1$, $\lambda_2$, $\lambda_3 : \mathbb{C}_+\times\mathbb{C}_+  \rightarrow [0,1)$ such that for all $x,y\in X$ the following axioms holds:\\
\begin{enumerate}

  \item  $\lambda_1(x,y)+\lambda_2(x,y)+\lambda_3(x,y)<1$;\\
  \item the mapping $\gamma : \mathbb{C}_+\times\mathbb{C}_+\rightarrow [0,1)$ defined by $\gamma(x,y)=\frac{\lambda_1(x,y)}{1-\lambda_2(x,y)-\lambda_3(x,y)}$ belongs to $\Gamma$;\\
  \item $d(Sx,Ty)\preceq \lambda_1(x,y)d(x,y)+\lambda_2(x,y)\frac{d(Sx,x)d(Ty,y)}{1+d(x,y)}+\lambda_3(x,y)\frac{d(Sx,y)d(Ty,y)}{1+d(x,y)}$
\end{enumerate}
Then $S$ and $T$ have a unique common fixed point in $X$.
\end{corollary}
\begin{corollary}
Let $(X,d)$ be a C-complete complex valued metric space and $T:X\rightarrow X$ be a self mappings. If there exists three mappings $\lambda_1$, $\lambda_2$, $\lambda_3 : \mathbb{C}_+\times\mathbb{C}_+  \rightarrow [0,1)$ such that for all $x,y\in X$ the following axioms holds:\\
\begin{enumerate}

  \item  $\lambda_1(x,y)+\lambda_2(x,y)+\lambda_3(x,y)<1$;\\
  \item the mapping $\gamma : \mathbb{C}_+\times\mathbb{C}_+\rightarrow [0,1)$ defined by $\gamma(x,y)=\frac{\lambda_1(x,y)}{1-\lambda_2(x,y)-\lambda_3(x,y)}$ belongs to $\Gamma$;\\
  \item $d(Tx,Ty)\preceq \lambda_1(x,y)d(x,y)+\lambda_2(x,y)\frac{d(Tx,x)d(Ty,y)}{1+d(x,y)}+\lambda_3(x,y)\frac{d(Tx,y)d(Ty,y)}{1+d(x,y)}$
\end{enumerate}
Then $T$ have a unique fixed point in $X$.
\end{corollary}
\end{corollary}

\section{Urysohn integral equations}
In this section, we show that Theorem 3.1 can be applied to the existence and uniqueness of a common solution of the system of the Urysohn integral equations.
\begin{eqnarray}\label{eq17}
x(t)=\psi_i(t)+\int_{a}^{b}K_i(t,s,x(s))ds
\end{eqnarray}
where $i=1, 2, 3, 4, a,b\in \mathbb{R}$ with $a\leq b, t\in [a,b], z, \psi_i\in C([a,b],\mathbb{R}^n)$ and $K_i:[a,b]\times [a,b]\times \mathbb{R}^n\rightarrow \mathbb{R}^n$ is a given mapping for each $i=1, 2, 3, 4$.\\
Throughout in this section, for each $i=1, 2, 3, 4$ and $K_i$ in equation (\ref{eq17}), we shall use the following symbols;
\begin{eqnarray*}
\delta_i(x(t))=\int_{a}^{b}K_i(t,s,x(s))ds
\end{eqnarray*}

\begin{theorem}
Consider the Urysohn integral equation (\ref{eq17}). Assume that the following given conditions hold for each $t\in [a,b]$;\\
\begin{enumerate}
  \item $\psi_1(t)+\psi_4(t)+\delta_1(t)-\delta_4(\delta_1x(t)+\psi_1(t)+\psi_4(t))=0$
  and $\psi_2(t)+\psi_3(t)+\delta_2x(t)-\delta_3(\delta_2x(t)+\psi_2(t)+\psi_3(t))=0$\\
  \item $\psi_1(t)+3\psi_3(t)+\delta_1(\delta_1x(t)+\psi_1(t))+2\delta_3x(t)+\delta_3(2x(t)-\delta x(t)-\psi_3(t))=4x(t)$ and
   $\psi_2(t)+3\psi_4(t)+\delta_2(\delta_2x(t)+\delta_2(t))+2\delta_4x(t)+\delta_4(2x(t)-\delta_4x(t)-\psi_4(t))=4x(t)$
\end{enumerate}
If there exists three mappings $\lambda_1$, $\lambda_2$, $\lambda_3 : \mathbb{C}_+\times\mathbb{C}_+  \rightarrow [0,1)$ such that, for all $x,y\in (\mathbb{C}_{+}\times\mathbb{C}_{+}) $, satisfying the following:\\
$(i)$ $\lambda_1(x,y)+2\lambda_2(x,y)+2\lambda_3(x,y)<1$;\\
$(ii)$ the mapping $\gamma : \mathbb{C}_+\times\mathbb{C}_+\rightarrow [0,1)$ defined by $\gamma(x,y)=\frac{\lambda_1(x,y)+\lambda_2(x,y)+\lambda_3(x,y)}{1-\lambda_2(x,y)-\lambda_3(x,y)}$ belongs to $\Gamma$;\\
$(iii)$ for each $x,y\in X$ and $t\in [a,b]$, we have \\
\begin{eqnarray*}
A_{xy}(t)\sqrt{1+a^2}e^{itan^{-1}a}&\preceq& \lambda_1\big(E_{x}(t),F_{y}(t)\big) B_{xy}(t)+\lambda_2\big(E_{x}(t),F_{y}(t)\big)C_{xy}(t)
+\lambda_3\big(E_{x}(t),F_{y}(t)\big)D_{xy}(t)
\end{eqnarray*}
where
\begin{eqnarray*}
A_{xy}(t)&=&\biggl|\biggl|\delta_1x(t)+\psi_1(t)-\delta_2x(t)-\delta_2(t)\biggr|\biggr|_{\infty};\\
B_{xy}(t)&=&\biggl(\biggl|\biggl|2x(t)-\delta_3x(t)-\psi_3(t)-2y(t)+\delta_4y(t)+\psi_4(t)\biggr|\biggr| \biggr)_{\infty}\sqrt{1+a^2}e^{itan^{-1}a};\\
C_{xy}(t)&=&\biggl(\biggl|\biggl|2x(t)-\delta_3x(t)-\psi_3(t)-\delta_1x(t)-\psi_2(t)\biggr|\biggr|_{\infty}+\biggl|\biggl|2y(t)-\delta_4y(t)-\psi_4(t)-\delta_2y(t)-\psi_2(t)\biggr|\biggr|_{\infty}\biggr)\sqrt{1+a^2}e^{itan^{-1}a}\\
D_{xy}(t)&=&\biggl(\biggl|\biggl|2x(t)-\delta_3x(t)-\psi_3(t)-\delta_2y(t)-\psi_2(t)\biggr|\biggr|_{\infty}+\biggl|\biggl|2y(t)-\delta_4y(t)-\psi_4(t)-\delta_1x(t)-\psi_1(t)\biggr|\biggr|_{\infty}\biggr)\sqrt{1+a^2}e^{itan^{-1}a}\\
E_{x}(t)&=&\max_{t\in [a,b]}||2x(t)-\delta_3x(t)-\psi_3(t)||_{\infty}\sqrt{1+a^2}e^{itan^{-1}a};\\
F_{y}(t)&=&\max_{t\in [a,b]}||2y(t)-\delta_4y(t)-\psi_4(t)||_{\infty}\sqrt{1+a^2}e^{itan^{-1}a}
\end{eqnarray*}
Then the system of equations (\ref{eq17}) have a unique commom solution.
\end{theorem}

\begin{proof}
let $X=C([a,b],\mathbb{R}^n)$ and $d:X\times X\rightarrow\mathbb{C}$ be defined by \\
\begin{eqnarray}\label{eq18}
d(x,y)=\max_{t\in [a,b]}||x(t)-y(t)||_{\infty}\sqrt{1+a^2}e^{itan^{-1}a}
\end{eqnarray}
Then $(X,d)$ be a C-complete complex valued metric space.

Define mappings $S, T, f$ and $g:X\rightarrow X$ by
\begin{eqnarray*}
Sx(t)&=&\delta_1x(t)+\psi_1(t)=\int_{a}^{b}K_1(t,s,x(s))ds+\psi_1(t);\\
Tx(t)&=&\delta_2x(t)+\psi_2(t)=\int_{a}^{b}K_2(t,s,x(s))ds+\psi_2(t);\\
fx(t)&=&2x(t)-\delta_3x(t)-\psi_3(t)=2x(t)-\int_{a}^{b}K_3(t,s,x(s))ds+\psi_3(t);\\
gx(t)&=&2x(t)-\delta_4x(t)-\psi_4(t)=2x(t)-\int_{a}^{b}K_4(t,s,x(s))ds+\psi_4(t);
\end{eqnarray*}
let $x,y\in X$, Then we get
\begin{eqnarray}\label{eq19}
\begin{aligned}
d(fx,0)&=\max_{t\in[a,b]}||2x(t)-\delta_3x(t)-\psi_3(t)||_{\infty}\sqrt{1+a^2}e^{itan^{-1}a};\\
d(gx,0)&=\max_{t\in[a,b]}||2x(t)-\delta_4x(t)-\psi_4(t)||_{\infty}\sqrt{1+a^2}e^{itan^{-1}a};\\
d(Sx,Ty)&=\max_{t\in[a,b]}||\delta_1x(t)+\psi_1(t)-\delta_2y(t)-\psi_2(t)||_{\infty}\sqrt{1+a^2}e^{itan^{-1}a};\\
d(fx,gy)&=\max_{t\in[a,b]}||2x(t)-\delta_3x(t)-\psi_3(t)-2y(t)+\delta_4y(t)-\psi_4(t)||_{\infty}\sqrt{1+a^2}e^{itan^{-1}a};\\
d(fx,Sx)&=\max_{t\in[a,b]}||2x(t)-\delta_3x(t)-\psi_3(t)-\delta_1x(t)-\psi_1(t)||_{\infty}\sqrt{1+a^2}e^{itan^{-1}a};\\
d(gy,Ty)&=\max_{t\in[a,b]}||2y(t)-\delta_4y(t)-\psi_4(t)-\delta_2x(t)-\psi_2(t)||_{\infty}\sqrt{1+a^2}e^{itan^{-1}a};\\
d(fx,Ty)&=\max_{t\in[a,b]}||2x(t)-\delta_3x(t)-\psi_3(t)-\delta_2y(t)i\psi_2(t)||_{\infty}\sqrt{1+a^2}e^{itan^{-1}a};\\
d(gy,Sx)&=\max_{t\in[a,b]}||2y(t)-\delta_4y(t)-\psi_4(t)-\delta_1x(t)-\psi_1(t)||_{\infty}\sqrt{1+a^2}e^{itan^{-1}a}
\end{aligned}
\end{eqnarray}
From condition (iii) of Theorem (4.1), for each $t\in [a,b]$, we have
\begin{eqnarray*}
A_{xy}(t)\sqrt{1+a^2}e^{itan^{-1}a}&\preceq& \lambda_1\big(E_{x}(t),F_{y}(t)\big)B_{xy}(t)+\lambda_2\big(E_{x}(t),F_{y}(t)\big)C_{xy}(t)
+\lambda_3\big(E_{x}(t),F_{y}(t)\big)D_{xy}(t)\\
&\preceq& \lambda_1\big(E_{x}(t),F_{y}(t)\big)\max_{t\in [a,b]} B_{xy}(t)+\lambda_2\big(E_{x}(t),F_{y}(t)\big)\max_{t\in [a,b]}C_{xy}(t)
+\lambda_3\big(E_{x}(t),F_{y}(t)\big)\max_{t\in [a,b]}D_{xy}(t)
\end{eqnarray*}
This implies,
\begin{eqnarray*}
\max_{t\in [a,b]}A_{xy}(t)\sqrt{1+a^2}e^{itan^{-1}a}&\preceq& \lambda_1\big(E_{x}(t),F_{y}(t)\big)\max_{t\in [a,b]}B_{xy}(t)
+\lambda_2\big(E_{x}(t),F_{y}(t)\big)\max_{t\in [a,b]}C_{xy}(t)
+\big(E_{x}(t),F_{y}(t)\big)\max_{t\in [a,b]}D_{xy}(t)\\
\end{eqnarray*}
Using equations (\ref{eq19}), we get
\begin{eqnarray*}
d(Sx,Ty)\preceq \lambda_1(fx,gy)d(fx,gy)+\lambda_2(fx,gy)(d(fx,Sx)+d(gy,Ty))+\lambda_3(fx,gy)(d(fx,Ty)+d(gy,Sx))
\end{eqnarray*}
Next, we shall show that $S(X)\subseteq g(X)$. For this consider
\begin{eqnarray*}
g\bigl(Sx(t)+\psi_4(t)\bigr)&=& 2\bigl[Sx(t)+\psi_4(t)\bigr]-\delta_4\bigl[Sx(t)+\psi_4(t)\bigr]-\psi_4(t)\\
&=& Sx(t)+Sx(t)+\psi_4(t)-\delta_4\bigl(Sx(t)+\psi_4(t)\bigr)\\
&=& Sx(t)+\delta_1x(t)+\psi(t)+\psi_4(t)-\delta_4\bigl(\delta_1x(t)+\psi_1(t)+\psi_4(t)\bigr)\\
&=& Sx(t)+\psi_1(t)+\psi_4(t)+\delta_1x(t)-\delta_4\bigl(\delta_1x(t)+\psi_1(t)+\psi_4(t)\bigr)\\
\end{eqnarray*}
Using given condition (1) of Theorem 4.1, we get $g\bigl(Sx(t)+\psi_4(t)\bigr)=Sx(t)$. which shows that $S(X)\subseteq g(X)$. Similary
we can show that $T(x)\subseteq f(X)$.\\
Next, we shall show that the pair $(S,f)$ and $(T,g)$ are weakly compatible. For each $t\in[a,b]$, we get
\begin{eqnarray}\label{eq20}
\bigl|\bigl|fSx(t)-Sfx(t)\bigr|\bigr|&=&\bigl|\bigl|f\bigl(\delta_1x(t)+\psi_1(t)\bigr)-S\bigl(2x(t)-\delta_3x(t)-\psi_3(t)\bigr|\bigr|\nonumber\\
&=&\bigl|\bigl|2\bigl(\delta_1x(t)+\psi_1(t)\bigr)-\delta_3\bigl(\delta_1x(t)+\psi_1(t)\bigr)-\psi_3(t)-\delta_1\bigl(2x(t)-\delta_3x(t)-\psi_3(t)\bigr)-\psi_1(t)\bigr|\bigr|
\end{eqnarray}
If $Sx=fx$ for some $x\in X$, then we have
\begin{eqnarray*}
\delta_1x(t)+\psi_1(t)=2x(t)-\delta_3x(t)-\psi_3(t) \textit{ for all $t\in [a,b]$ }
\end{eqnarray*}
Therefore from equation (\ref{eq20}), we get
\begin{eqnarray*}
\bigl|\bigl|fSx(t)-Sfx(t)\bigr|\bigr|=\bigl|\bigl|4x(t)-2\delta_3x(t)-3\psi_3(t)-\delta_3\bigl(2x(t)-\delta_3(t)-\psi_3x(t)-\psi_3(t)\bigr)-\delta_1
\bigl(\delta_1x(t)+\psi_1(t)\bigr)-\psi_1(t)\bigr|\bigr|
\end{eqnarray*}
for all $t\in [a,b]$. From condition (2), we get $\bigl|\bigl|fSx(t)-Sfx(t)\bigr|\bigr|=0$, that is, $fSx(t)=Sfx(t)$ for all $t\in [a,b]$. Therefore,
$fSx=Sfx$ whenever $Sx=fx$. Hence the pair $(S,f)$ is weakly compatible. Similarly, we can show that $(T,g)$ is weakly compatible.\\
Thus all the conditions of Theorem (3.1) are satisfied. Therefore there exists a unique common fixed point of $ S,T,f,g$ in $X$ and consequently there exist a unique common solution of the system $(\ref{eq17}).$
\end{proof}


\end{document}